\documentclass[12pt,reqno]{amsart}

\usepackage{fullpage}
\usepackage{amssymb}    
\usepackage[all]{xy}    
\usepackage{verbatim}
\usepackage{mathdots}   
\usepackage{enumitem}   

\theoremstyle{plain} 
\newtheorem{theorem}{Theorem}[section]

\newtheorem{lemma}[theorem]{Lemma}

\newtheorem{proposition}[theorem]{Proposition}
\newtheorem{corollary}[theorem]{Corollary}

\newtheorem{problem}[theorem]{Problem}
\theoremstyle{remark} 
\newtheorem{example}[theorem]{Example}

\numberwithin{equation}{section}

\newcommand{\seclabel}[1]{\label{sec:#1}} 
\newcommand{\thmlabel}[1]{\label{thm:#1}} 
\newcommand{\lemlabel}[1]{\label{lem:#1}} 
\newcommand{\corlabel}[1]{\label{cor:#1}} 
\newcommand{\prplabel}[1]{\label{prp:#1}} 
\newcommand{\exmlabel}[1]{\label{exm:#1}} 
\newcommand{\eqnlabel}[1]{\label{eqn:#1}} 

\newcommand{\thmref}[1]{\ref{thm:#1}} 
\newcommand{\lemref}[1]{\ref{lem:#1}} 
\newcommand{\prpref}[1]{\ref{prp:#1}} 
\newcommand{\eqnref}[1]{\eqref{eqn:#1}} 

\newcommand{\DDD}{\mathcal{D}}
\newcommand{\RRR}{\mathcal{R}}
\newcommand{\LLL}{\mathcal{L}}
\newcommand{\greenD}{\thinspace\mathcal{D}\thinspace}
\newcommand{\greenL}{\thinspace\mathcal{L}\thinspace}
\newcommand{\greenR}{\thinspace\mathcal{R}\thinspace}

\newcommand{\defn}[1]{\emph{\textbf{#1}}}  

\newcommand{\mypara}{\ /\!/\ }

\author{Michael Kinyon}
\address{Department of Mathematics \\
University of Denver \\
Denver, CO 80208 USA}
\email{mkinyon@du.edu}

\author{Jonathan Leech}
\address{Department of Mathematics\\
Westmont College \\
955 La Paz Road\\
Santa Barbara, CA 93108 USA}
\email{leech@westmont.edu}

\title{Categorical Skew Lattices}
\date{Last updated: \today}

\begin{document}

\begin{abstract}
Categorical skew lattices are a variety of skew lattices 		
on which the natural partial order is especially well behaved.  While
most skew lattices of interest are categorical, not all are.  They are	
characterized by a countable family of forbidden subalgebras.  We
also consider the subclass of strictly categorical skew lattices.
\end{abstract}

\maketitle

\section{Introduction and Background}
\seclabel{intro}

A \defn{skew lattice} is an algebra $\mathbf{S}=(S;\lor,\land)$ where $\lor$ and $\land$ are associative, idempotent binary operations satisfying the absorption identities
\begin{equation}
\eqnlabel{absorb}
x\land (x\lor y) = x = (y\lor x)\land x
\qquad\text{and}\qquad
x\lor (x\land y) = x = (y\land x)\lor x\,.
\end{equation}
Given that $\lor$ and $\land$ are associative and idempotent, \eqnref{absorb} is equivalent to the dualities:
\begin{equation}
\eqnlabel{dualities}
x\land y = x\quad\text{iff}\quad x\lor y = y
\qquad\text{and}\qquad x\land y = y\quad\text{iff}\quad x\lor y = x\,.
\end{equation}	
Every skew lattice has a \defn{natural preorder} (or quasi-order) defined by
\begin{equation}
\eqnlabel{preorder}
x \succeq y\quad \Leftrightarrow\quad  x\lor y\lor x = x\quad\text{or equivalently}\quad
y\land x\land y = y\,.
\end{equation}
The \defn{natural partial order} is defined by
\begin{equation}
\eqnlabel{order}
x \geq y\quad \Leftrightarrow\quad  x\lor y = x = y\lor x\quad\text{or equivalently}\quad
x\land y = y = y\land x\,.
\end{equation}
The latter refines the former in that $x\succeq y$ implies $x\geq y$ but not conversely. 
In what follows, any mentioned preordering or partial ordering of a skew lattice is assumed to be natural. Of course
$x > y$ means $x \geq y$ but $x\neq y$; likewise, $x \succ y$ means $x \succeq y$ but not $y \succeq x$.

A natural model of a skew lattice is given by any set of idempotents $S$ in a ring $R$ that is closed under
the operations $\land$ and $\lor$ defined in terms of addition and multiplication by $x\land y = xy$ and
$x\lor y = x + y - xy$. Another natural model is given by the set of all partial functions $\mathcal{P}(X,Y)$
from a set $X$ to a set $Y$, where for partial functions $f,g\in \mathcal{P}(X,Y)$, $f\land g =
g|_{\mathrm{dom}(f)\cap \mathrm{dom}(g)}$ and $f\lor g = f \cup g|_{\mathrm{dom}(g)\backslash \mathrm{dom}(f)}$.

Every skew lattice is \defn{regular} in that the identity
$x\circ y\circ x\circ z\circ x = x\circ y\circ z\circ x$ holds for both $\circ = \lor$ and $\circ = \land$ (see \cite[Theorem 1.15]{Leec89} or \cite[Theorem 1.11]{Leec96}). As a consequence, one quickly gets:
\begin{subequations}
\begin{alignat}{2}
x\lor y\lor x'\lor z\lor x'' &= x\lor y\lor z\lor x'' &&\qquad\text{if}\quad x'\preceq x, x''
\eqnlabel{regular1} \\
\intertext{and}
x\land y\land x'\land z\land x'' &= x\land y\land z\land x'' &&\qquad\text{if}\quad x' \succeq x, x''\,.
\eqnlabel{regular2}
\end{alignat}
\end{subequations}

In any lattice, $\geq$ and $\succeq$ are identical, with $\lor$ and $\land$ determined by $s\lor y= \sup\{x,y\}$ and $x\land y = \inf\{x,y\}$. For skew lattices, the situation is more complicated. To see what happens, we must first recall several fundamental aspects of skew lattices. The preorder $\succeq$ induces a natural equivalence $\DDD$ defined by $x\greenD y$ if $x\succeq y\succeq x$. This is one of three \defn{Green's relations} defined by:
\begin{alignat*}{3}
x\greenR y &\ \Leftrightarrow\ & (x\land y = y\ \&\ y\land x = x) &\ \Leftrightarrow\ &
(x\lor y = x\ \&\ y\lor x = y)\,.     \tag{$\RRR$} \\
x\greenL y &\ \Leftrightarrow\ & (x\land y = x\ \&\ y\land x = y)&\ \Leftrightarrow\ &
(x\lor y = y\ \&\ y\lor x = x)\,.     \tag{$\LLL$} \\
x\greenD y &\ \Leftrightarrow\ & (x\land y\land x = x\ \&\ y\land x\land y = y) &\ \Leftrightarrow\ &
(x\lor y\lor x = x\ \&\ y\lor x\lor y = y)\,.     \tag{$\DDD$}
\end{alignat*}
$\RRR$, $\LLL$ and $\DDD$ are congruences on any skew lattice, with
$\greenL\lor \greenR = \greenL\circ \greenR = \greenR\circ \greenL = \greenD$ and $\greenL\cap \greenR = \Delta$, the identity equivalence. Their congruence classes (called $\greenR$-classes, $\greenL$-classes or $\greenD$-classes) are all rectangular subalgebras. (A skew lattice is \defn{rectangular} if $x\land y\land x = x$, or equivalently, $x\lor y\lor x = x$, or also equivalently, $x\land y = y\lor x$ holds. These are precisely the anti-commutative skew lattices in that $x\land y = y\land x$ or $x\lor y = y\lor x$ imply $x = y$. See \cite[{\S}1]{Leec89}
or recently, \cite[{\S}1]{KL}.) The Green's congruence classes of a an element $x$ are denoted, respectively, by $\mathcal{R}_x$, $\mathcal{L}_x$ or $\mathcal{D}_x$.

The First Decomposition Theorem for Skew Lattices \cite[Theorem 1.7]{Leec89} states: \emph{Given a skew lattice} $\mathbf{S}$, \emph{each} $\DDD$-\emph{class is a maximal rectangular subalgebra of} $\mathbf{S}$ \emph{and} $\mathbf{S}/\DDD$ \emph{is the maximal lattice image of} $\mathbf{S}$.  In brief, \emph{every skew lattice is a lattice of rectangular [anticommutative] subalgebras} in that it looks roughly like a lattice whose points are rectangular skew lattices.
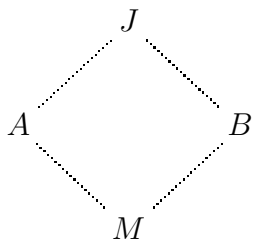
\begin{figure}[h!]
\begin{centerline}
{\xymatrix{
                & J                         &                       \\
A \ar@{.}[ur]   &                           & B \ar@{.}[ul]   \\
                & M \ar@{.}[ul] \ar@{.}[ur] &
}}
\end{centerline}
\caption{$A$, $B$, $J$, \& $M$ are maximal rectangular subalgebras}
\end{figure}
Clearly $x \succeq y$ in $\mathbf{S}$ if and only if $D_x \geq D_y$ in the lattice $\mathbf{S}/\DDD$ where $D_x$ and $D_y$ are the $\DDD$-classes of $x$ and $y$, respectively. Given $a\in A$ and $b\in B$ for $\DDD$-classes $A$ and $B$, $a\lor b$ just lie in their join $\DDD$-class $J$; similarly $a\land b$ must lie in their meet $\DDD$-class $M$.

Our interest in this paper is in \defn{skew chains} that consist of totally ordered families of $\DDD$-classes: $A > B > \cdots > X$. As a (sub-)skew lattice, a skew chain $T$ is \defn{totally preordered}: given $x,y\in T$, either $x\preceq y$ or $y\preceq x$. Of special interest are skew chains of length $1$ ($A>B$) called \defn{primitive skew lattices}, and skew chains of length $2$ ($A>B>C$) that occur in skew lattices.

Given a primitive skew lattice with $\DDD$-class structure $A>B$, an $\mathbf{A}$-\defn{coset in} $\mathbf{B}$ is any subset of $B$ of the form
\[
A\land b\land A = \{a\land b\land a'\mid a, a'\in A\} = \{a\land b\land a\mid a\in A\}
\]
for some $b \in B$. (The second equality follows from \eqnref{regular2}.)  Any two $A$-cosets in $B$ are either identical or else disjoint. Since $b$ must lie in $A\land b\land A$ for all $b \in B$, \emph{the} $A$-\emph{cosets in} $B$ \emph{form a partition of} $B$.  Dually a $B$-\defn{coset in} $A$ is a subset of $A$ of the form
\[
B\lor a\lor B = \{b\lor a\lor b\mid  b, b'\in B\} = \{b\lor a\lor b\mid b\in B\}
\]
for some $a \in A$.  Again, the $B$-cosets in $A$ partition $A$. Given a $B$-coset $X$ in $A$ and an $A$-coset $Y$ in $B$, the natural partial ordering induces a \defn{coset bijection} $\varphi : X\to Y$ given by $\varphi(a) = b$ for $a\in X$ and $b\in Y$ if and only if $a > b$, in which case $b = \varphi(a) = a\land y\land a$ for \emph{all} $y\in Y$ and $a = \varphi^{-1}(b) = b\lor x\lor b$ for \emph{all} $x\in X$. Cosets are rectangular subalgebras of their $\DDD$-classes; moreover, all coset bijections are isomorphisms between these subalgebras.  \emph{All} $A$-\emph{cosets in} $B$ \emph{and all} $B$-\emph{cosets in} $A$ \emph{thus share a common size and structure}.  If $a, a'\in A$ lie in a common $B$-coset, we denote this by $a -_B a'$; likewise $b -_A b'$ in $B$ if $b$ and $b'$ lie in a common $A$-coset.  This is illustrated in the partial configuration below where $\ddots$ and $\iddots$ indicate $>$ between $a$'s and $b$'s. (The coset bijections from $\{a_1, a_2\}$ to $\{b_3, b_4\}$ and from $\{a_5, a_6\}$ to $\{b_1, b_2\}$ are not shown.)
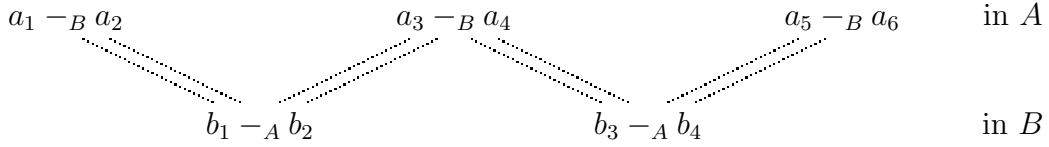
\[
{\xymatrix{
a_1 -_B a_2\ar@{.}[dr]!<2ex,0ex>&  & a_3 -_B a_4 \ar@{.}[dl]!<-2ex,0ex> \ar@{.}[dr]!<2ex,0ex>& & a_5 -_B a_6 \ar@{.}[dl]!<-2ex,0ex>  & \text{in } A \\
   & b_1 -_A b_2\ar@{.}[ul]!<-2ex,0ex> \ar@{.}[ur]!<2ex,0ex>&               & b_3 -_A b_4\ar@{.}[ul]!<-2ex,0ex> \ar@{.}[ur]!<2ex,0ex> &       & \text{in } B
}}
\]
\noindent \emph{Cosets and their bijections determine} $\lor$ \emph{and} $\land$ \emph{in this situation.} Given $a\in A$ and $b\in B$:
\begin{subequations}
\begin{align}
a\lor b &= a\lor a' \text{ and } b\lor a = a'\lor a\text{ in } A\text{ where } a'-_B a\text{ is such that } a' \geq b\,.
\eqnlabel{determine1}\\
a\land b &= b'\land b\ \text{  and } b\land a = b\land b\ \text{ in } B\text{ where } b'-_A b\ \text{ is such that } a \geq b\,.
\eqnlabel{determine2}
\end{align}
\end{subequations}
(See \cite[Lemma 1.3]{Leec93}.) This explains how $\geq$ determines $\lor$ and $\land$ in the primitive case. How this is extended to the general case where $A$ and $B$ are incomparable $\DDD$-classes is explained in \cite[{\S}3]{Leec93}; see also \cite{Leec96}.

This paper focuses on skew chains of $\mathcal{D}$-classes $A > B > C$ in a skew lattice and their three primitive subalgebras: $A>B$, $B>C$ and $A>C$. Viewing coset bijections as partial bijections between the relevant $\DDD$-classes one may ask: \emph{is the composite} $\psi\varphi$ \emph{of coset bijections} $\varphi:A\to B$ \emph{and} $\psi:B\to C$, \emph{if nonempty, a coset bijection from} $A$ \emph{to} $C$? If the answer is always yes, the skew chain is called \defn{categorical}. (Since including identity maps on $\DDD$-classes and empty partial bijections if needed creates a category with $\DDD$-classes for objects, coset bijections for morphisms and composition being that of partial bijections.) If this occurs for all skew chains in a skew lattice $\mathbf{S}$, then $\mathbf{S}$ is categorical. If such compositions are also always nonempty, the skew chain [skew lattice] is \defn{strictly categorical}.

Both categorical and strictly categorical skew lattices form varieties. (See \cite[Theorem 3.16]{Leec93} and Corollary 4.3 below.) We will see that distributive skew lattices are categorical, and in particular skew lattices in rings are categorical. All skew Boolean algebras \cite{Leec90} are strictly categorical. Categorical skew lattices were introduced in \cite{Leec93}. Here we take an alternatively approach.

In all this, individual ordered pairs $a>b$ are bundled to form coset bijections. We first look at how this ``bundling'' process (parallelism) extends from the $A-B$ and $B-C$ settings to the $A-C$ settings in the next section.

\section{Parallel ordered pairs}
\seclabel{parallel}

Suppose $A > B$ is a (primitive) skew chain and $\varphi : X\to Y$ is a fixed coset bijection where $X$ is a $B$-coset in $A$ and $Y$ is an $A$-coset in $B$. Viewing the function $\varphi$ as a binary relation, let us momentarily identify it with the set of strictly ordered pairs $a > b$ where $a\in X$, $b\in Y$ are such that $\varphi(a) = b$. Suppose $a > b$ and $a' > b'$ are two such pairs. Since $b' = \varphi(a') = a'\land y\land a'$ for all $y\in Y$, we certainly have $b' = a'\land b\land a$ and similarly $b = a\land b'\land a$. Since $a' = \varphi^{-1}(b') = b'\lor x\lor b'$ for all $x\in X$, we have $a' = b'\lor a\lor b'$ and similarly $a = b\lor a'\lor b$. These observations motivate the following definition.

Strictly ordered pairs $a > b$ and $a' > b'$ in a skew lattice $\mathbf{S}$ are said to be \defn{parallel}, denoted $a > b\mypara a' > b'$, if $a\greenD a'$, $b\greenD b'$, $a' = b'\lor a\lor b'$ and $b' = a'\land b\land a'$. In this case, \eqnref{regular1} and \eqnref{regular2} imply that $a = b\lor a'\lor b$ and $b = a\land b'\land a$ also, so that the concept is symmetric with respect to both inequalities. In fact, the two pairs are parallel precisely when both lie in a common coset bijection $\varphi$, when considered to be a binary relation. Indeed, $a > b\mypara a' > b'$ implies that both $a$ and $a'$ share a common $\mathcal{D}_b$-coset in $\mathcal{D}_a$, and $b$ and $b'$ share a common $\mathcal{D}_a$-coset in $\mathcal{D}_b$, making both pairs belong to a common $\varphi$. Conversely, if $a > b$ and $a' > b'$ lie in a common coset bijections so that $a,a'$ share a $\mathcal{D}_b$-coset in $\mathcal{D}_a$ and $b,b'$ share a $\mathcal{D}_a$-coset in $\mathcal{D}_b$, then $a'=b'\lor a\lor b'$ and $b'=a'\land b\land a'$ must follow so that $a > b\mypara a'> b'$. Thus:

\begin{proposition}
\prplabel{2.1}
Parallelism is an equivalence relation on the set of all partially ordered pairs $a > b$ in a skew lattice $\mathbf{S}$, the equivalence classes of which form coset bijections when the latter are viewed as binary relations. Moreover:
\begin{enumerate}[label=\emph{\roman*)}]
\item If $a > b\mypara a' > b'$, then $a = a'$ if and only if $ b = b'$;
\item If $a > b\mypara a' > b'$ and $b > c\mypara b' > c'$, then $a > c\mypara a' > c'$;
\item Given just $a\succ b$, then $a > a\land b\land a\mypara b\lor a\lor b > b$.
\end{enumerate}
\end{proposition}
\begin{proof}
The first claim is routine, and (i)-(iii) follow from basic properties of coset bijections: their being bijections indeed, their composition and their connections to their particular cosets of relevance.
\end{proof}

Now we return to the point of view that for a skew chain $A > B$, a coset bijection $\varphi : X\to Y$, $X\subseteq A$, $Y\subseteq B$, is a partial bijection $\varphi : A\to B$ of the $\mathcal{D}$-classes. Let $A > B > C$ be a $3$-term skew chain and suppose $\varphi : A\to B$ and $\psi : B\to C$ are coset (partial) bijections. Suppose that the composite partial bijection $\psi\circ \varphi : A\to C$ is nonempty, say $a > b > c$ with $b =\varphi(a)$ and $c =\psi(b)$. Then there is a uniquely determined partial bijection $\chi : A \to C$ defined on its coset domain by $\chi(u) = u\land c\land u$ such that $\psi\circ \varphi \subseteq \chi$. Later we shall see instances where the inclusion is proper. We are interested in characterizing equality.

In terms of parallelism and the fixed triple $a > b > c$, the situation we have described so far is that \emph{if} $a > b\mypara a'> b'$ and $b > c\mypara b' > c'$, \emph{then} $a > c\mypara a'> c'$. We see that $\chi = \psi \circ \varphi$ precisely when the converse holds, that is, \emph{if} $a > c\mypara a'> c'$, \emph{then} there exists a (necessarily) unique $b'\in B$ such that $a > b\mypara a' > b'$ and $b > c\mypara b' > c'$. In particular, $b'$ must equal both $a'\land b\land a'$ and $c'\lor b\lor c'$. This gives the following Hasse configuration of parallel pairs.
\begin{equation}
\eqnlabel{parallel}
\begin{array}{lcc}
a & \text{--}\text{--} & a' = c'\lor a\lor c' = b'\lor a\lor b' \\
\vdots & & \vdots \\
b & \text{--}\text{--} & b' = c'\lor b\lor c' = a'\land b\land a'\\
\vdots & & \vdots \\
c & \text{--}\text{--} & c' = a'\land c\land a' = b'\land c\land b'
\end{array}
\end{equation}

Now considering this for all possible coset bijections in a skew lattice, we obtain the following characterization.

\begin{proposition}
\prplabel{orderfactorable}
A skew lattice $\mathbf{S}$ is categorical if and only if, given $a > b > c$ with $a > c\mypara a' > c'$, there exists a unique $b'\in S$ such that $a > b\mypara a' > b'$ and $b > c\mypara b' > c'$.
\end{proposition}

\begin{theorem}
\thmlabel{variety}
For a skew lattice $\mathbf{S}$, the following are equivalent.
\begin{enumerate}[label=\emph{\roman*)}]
\item $\mathbf{S}$ is categorical;
\item For all $x,y,z\in S$,
\begin{equation}
\eqnlabel{catimp1}
x\geq y\succeq z\quad\Rightarrow\quad
x\land (z\lor y\lor z)\land x = (x\land z\land x)\lor y\lor (x\land z\land x)\,;
\end{equation}
\item For all $x,y,z\in S$,
\begin{equation}
\eqnlabel{catimp2}
x\succeq y\geq z\quad\Rightarrow\quad
z\lor (x\land y\land x)\lor z = (z\lor x\lor z)\land y\land (z\lor x\lor z)\,.
\end{equation}
\end{enumerate}
\end{theorem}
\begin{proof}
Assume (i) holds and let $a\geq b\succeq c$ be given. If $a = b$ or if $b\greenD c$, then their insertion into \eqnref{catimp1} produces a trivial identity. Thus we may assume the comparisons to be strict: $a > b\succ c$. Proposition \prpref{2.1}(iii) gives $a > a\land c\land a\mypara c\lor a\lor c > c$. Since $c\lor a\lor c> c\lor b\lor c > c$, \eqnref{parallel} gives
\[
a\land (c\lor b\lor c\land a = (a\land c\land a)\lor c\lor b\lor c\lor (a\land c\land a)\,.
\]
From $c\greenD a\land c\land a$, \eqnref{regular1} reduces the right side to $(a\land c\land a)\lor b\lor (a\land c\land a)$ and so \eqnref{catimp1} holds. We have established (i)$\Rightarrow$(ii).

Conversely assume that (ii) holds, and let both $a > c\mypara a' > c'$ and $a > b > c$. Since $b > c\greenD c'$, $b\succ c'$. Thus $a > b\succ c'$, and so by \eqnref{catimp1},
\[
a\land (c'\lor b\lor c')\land a = (a\land c'\land a)\lor v\lor (a\land c'\land a) = c\lor b\lor c = b\,,
\]
since $a > b$ and $a\land c'\land a = c$. Taking two-sided meets with $a'$ gives
\begin{align*}
a'\land b\land a' &= a'\land a\land (c'\lor b\lor c')\land a\land a' && \\
&= a'\land a\land a'\land (c'\lor b\lor c')\land a'\land a\land a' && \text{(by regularity)} \\
&= a'\land (c'\lor b\lor c')\land a' &&
\text{(since }a\greenD a'\text{)} \\
&= (c'\lor a\lor c')\land (c'\lor b\lor c')\land (c'\lor a\lor c') &&
\text{(since }a' > c'\text{)} \\
&= (c'\lor a\lor b\lor c')\land (c'\lor b\lor c')\land (c'\lor b\lor a\lor c') && \text{(by }\eqnref{regular1}\text{)} \\
&= (c'\lor a\lor c'\lor b\lor c')\land (c'\lor b\lor c')\land (c'\lor b\lor c'\lor a\lor c') && \text{(by }\eqnref{regular1}\text{)} \\
&= c'\lor b\lor c && \text{(by }\eqnref{absorb}\text{)}\,.
\end{align*}
Thus \eqnref{parallel} holds and $\mathbf{S}$ is categorical.

We have established (i)$\Leftrightarrow$(ii). The proof of (i)$\Leftrightarrow$(iii) is dual to this, exchanging $\land$ and $\lor$ as needed.
\end{proof}

Next we will show that categorical skew lattices form a variety by giving characterizing identities. This was already done in \cite[Theorem 3.16]{Leec93}, but the identity given there is rather long. Here we give two new ones, the first being the shortest we know and the second exhibiting a certain amount of symmetry in the variables. First we recall more basic notions.

A skew lattice is \defn{right-handed} [respectively, \defn{left-handed}] if it satisfies the identities
\begin{subequations}
\begin{align}
x\land y\land x = y\land x\quad\text{and}\quad x\lor y\lor x = x\lor y\,. \eqnlabel{rh} \\
[x\land y\land x = x\land y\quad\text{and}\quad x\lor y\lor x = y\lor x]\,. \eqnlabel{lh}
\end{align}
\end{subequations}
Equivalently, $x\land y=y$ and $x\lor y=x$ [$x\land y=x$ and $x\lor y=y$] hold in each $\mathcal{D}$-class, thus reducing $\mathcal{D}$ to $\mathcal{R}$ [or $\mathcal{L}$]. Useful right- and left-handed variants of \eqnref{rh} and \eqnref{lh} are
\begin{subequations}
\begin{align}
x\succeq x'\quad\Rightarrow\quad x\land y\land x' = y\land x'\quad\text{and}\quad x\lor y\lor x' = x\lor y\,; \eqnlabel{rh-var} \\
x\succeq x'\quad\Rightarrow\quad x'\land y\land x = x'\land y\quad\text{and}\quad x'\lor y\lor x = y\lor x'\,; \eqnlabel{lh-var}
\end{align}
\end{subequations}
The \emph{Second Decomposition Theorem} \cite[Theorem 1.15]{Leec89} states that \emph{given any skew lattice} $\mathbf{S}$, $\mathbf{S}/\mathcal{R}$ \emph{and} $\mathbf{S}/\mathcal{L}$ \emph{are its respective maximal left- and right-handed images, and} $\mathbf{S}$ \emph{is isomorphic to their fibred product (pullback)} $\mathbf{S}/\mathcal{R}\times_{\mathbf{S}/\mathcal{D}} \mathbf{S}/\mathcal{L}$ \emph{over their maximal lattice image under the map} $x\mapsto (\mathcal{R}_x, \mathcal{L}_x)$. Thus a skew lattice $\mathbf{S}$ \emph{belongs to a variety} $\mathcal{V}$ \emph{of skew lattices if and only if both} $\mathbf{S}/\mathcal{R}$ \emph{and} $\mathbf{S}/\mathcal{L}$ \emph{do}. (See also \cite{Cvet07,Leec96}.)

\begin{theorem}
\thmlabel{identities}
Let $\mathbf{S}$ be a skew lattice. The following are equivalent.
\begin{enumerate}[label=\emph{\roman*)}]
\item $\mathbf{S}$ is categorical.
\item For all $x,y,z\in S$,
\begin{equation}
\eqnlabel{catshort}
x\land [(x\land y\land z\land y\land x)\lor y\lor (x\land y\land z\land y\land x)]\land x = x\land y\land x\,.
\end{equation}
\item For all $x,y,z\in S$,
\begin{equation}
\eqnlabel{catsymm}
x\land [(x\land z\land x)\lor y\lor (x\land z\land x)]\land x =
x\land [(z\land x\land z)\lor y\lor (z\land x\land z)]\land x\,.
\end{equation}
\end{enumerate}
\end{theorem}
\begin{proof}
Assume first that $\mathbf{S}$ is a left-handed categorical skew lattice. Suppose (i) holds. By Theorem \thmref{variety}, $\mathbf{S}$ satisfies the left-handed version of \eqnref{catimp1}:
\begin{equation}
\eqnlabel{catimp-lh}
x\geq y\succeq z\quad\Rightarrow\quad
x\land (y\lor z) = y\lor (x\land z)\,.
\end{equation}
Note that $x\lor y\geq y\succeq (y\lor x)\land y\land z$. We may thus apply \eqnref{catimp-lh}. The right side becomes
\[
y\lor [(x\lor y)\land (y\lor x)\land y\land z] =
y\lor [(x\lor y)\land y\land z] = y\lor [y\land z] = y\,,
\]
using left-handedness and absorption. Therefore the identity
\begin{equation}
\eqnlabel{cat1}
(x\lor y)\land [y\lor ((y\lor x)\land y\land z)] = y
\end{equation}
holds. Taking the meet of both sides on the left with $x$, we get
\begin{equation}
\eqnlabel{cat2}
x\land [y\lor ((y\lor x)\land y\land z)] = x\land y\,.
\end{equation}
Now replace $y$ with $y\land x$. The left side of \eqnref{cat2} becomes
\[
x\land [(y\land x)\lor (((y\land x)\lor x)\land y\land z)] = x\land [(y\land x)\lor (x\land y\land z)]\,,
\]
and the right side becomes $x\land y\land x = x\land y$. Thus we have the identity
\begin{equation}
\eqnlabel{cat3}
x\land [(y\land x)\lor (x\land y\land z)] = x\land y\,.
\end{equation}

Now meet both sides of \eqnref{cat3} on the left with $x\land (y\lor (x\land y\land z))$. On the right side, we get
\[
x\land (y\lor (x\land y\land z))\land x\land y = x\land (y\lor (x\land y\land z))\land y = x\land (y\lor (x\land y\land z))\,,
\]
since $y = y\lor y \succeq y\lor (x\land y\land z)$. The left side becomes
\begin{align*}
& x\land [y\lor (x\land y\land z)]\land x\land [(y\land x)\lor (x\land y\land z)] \\
&= x\land [y\lor (y\land x)\lor (x\land y\land z)]\land [(y\land x)\lor (x\land y\land z)] \\
&= x\land [(y\land x)\lor (x\land y\land z)] \\
&= x\land y\,,
\end{align*}
where the last step is an application of \eqnref{cat3}. Thus we have established
\begin{equation}
\eqnlabel{catshort-lh}
x\land (y\lor (x\land y\land z)) = x\land y\,,
\end{equation}
which is the left-handed version of \eqnref{catshort}. This proves (i)$\Rightarrow$(ii) for all left-handed skew lattices.

Continuing to assume $\mathbf{S}$ is left-handed, suppose (ii) holds. Replace $y$ with $y\lor z$ in \eqnref{catshort-lh}. On the left side, we obtain
\[
x\land (y\lor z\lor (x\land (y\lor z)\land z)) = x\land (y\lor z\lor (x\land z))\,.
\]
On the right side, we get $x\land (y\lor z)$, and so we have
\begin{equation}
\eqnlabel{cat4}
x\land (y\lor z\lor (x\land z)) = x\land (y\lor z)\,.
\end{equation}
Now in \eqnref{cat4}, replace $z$ with $z\land x$. On the left side, we get
\[
x\land (y\lor (z\land x)\lor (x\land z\land x))
= x\land (y\lor (z\land x\land z)\lor (x\land z))
= x\land (y\lor (x\land z))\,.
\]
On the right side, we get $x\land (y\lor (z\land x))$, and thus we obtain the identity
\begin{equation}
\eqnlabel{catsymm-lh}
x\land (y\lor (x\land z)) = x\land (y\lor (z\land x))\,,
\end{equation}
which is the left-handed version of \eqnref{catsymm}. This proves (ii)$\Rightarrow$(iii) in left-handed skew lattices.

Still assuming $\mathbf{S}$ is left-handed, suppose (iii) holds. Fix $a,b,c\in S$ satisfying $a\geq b\succeq c$. Then
\begin{align*}
a\land (b\lor c) &= a\land (b\lor (c\land a)) &&\text{(since } a\succeq c\text{)}\\
&= a\land (b\lor (a\land c)) &&\text{(by \eqnref{catsymm-lh})} \\
&= (a\lor (a\land c))\land (b\lor (a\land c)) && \\
&= (a\lor b\lor (a\land c))\land (b\lor (a\land c)) &&\text{(since }a\geq b\text{)}\\
&= b\lor (a\land c)\,.
\end{align*}
Thus \eqnref{catimp-lh} holds and so by Theorem \thmref{variety}, $\mathbf{S}$ is categorical. This proves (iii)$\Rightarrow$(i) for left-handed skew lattices.

In general, if $\mathbf{S}$ is a skew lattice, then conditions (i), (ii) and (iii) are equivalent for the maximal left-handed image $\mathbf{S}/\mathcal{R}$. The left-right (horizontal) dual of the whole argument implies that the same is true for $\mathbf{S}/\mathcal{L}$. It follows that (i), (ii) and (iii) are equivalent for $\mathbf{S}$ itself.
\end{proof}

\begin{corollary}
\corlabel{variety}
Categorical skew lattices form a variety.
\end{corollary}

Of course, categorical skew lattices are also characterized by the $\lor-\land$ duals of \eqnref{catshort} and \eqnref{catsymm}.

Recall that a skew lattice is \defn{distributive} if the following dual pair of identities holds:
\begin{align}
x\land (y\lor z)\land x &= (x\land y\land x)\lor (x\land z\land x)\,,
\eqnlabel{dist1} \\
x\lor (y\land z)\lor x &= (x\lor y\lor x)\land (x\lor z\lor x)\,.
\eqnlabel{dist2}
\end{align}
Many important classes of skew lattices are distributive, in particular, skew lattices in rings and skew Boolean algebras \cite{Bign95,Bign96,Cvet05a,Cvet05,Leec89,Leec90,Leec0x,Spin06}. Since \eqnref{dist1} implies \eqnref{catimp1}, we have:

\begin{corollary}
\corlabel{dist}
Distributive skew lattices are categorical.
\end{corollary}

\section{Forbidden subalgebras}
\seclabel{forbidden}

Clearly what occurs in the middle class of a $3$-term skew chain $A > B > C$ is significant. Two elements $b,b'\in B$ are $\mathbf{AC}$-\defn{connected} if a finite sequence $b = b_0, b_1,\ldots, b_n = b'$ in $B$ exists such that $b_i-_A b_{i+1}$ or $b_i-_C b_{i+1}$ for all $i\leq n-1$. A maximally $AC$-connected subset of $B$ is an $\mathbf{AC}$-\defn{component} of $B$ (or just \defn{component} if the context is clear). Given a component $B'$ in the middle class $B$, a sub-skew chain is given by $A > B' > C$. Indeed, if $A_1$ and $C_1$ are $B$-cosets in $A$ and $C$ respectively, then $A_1 > B' > C_1$ is an even smaller sub-skew chain.

Furthermore, let $X$ denote an $A$-coset in $B$ (thus $X = A\land b\land A$ for any $b\in X$) and let $Y$ denote a $C$-coset in $B$ (thus $Y = C\lor b\lor C$ for any $b\in Y$). If $X\cap Y \neq \emptyset$, it is called an $\mathbf{AC}$-\defn{coset} in $B$. \emph{When} $\mathbf{S}$ \emph{is categorical}, $(X\cap Y)\lor a\lor (X\cap Y)$ \emph{is a} $C$-\emph{coset in} $A$ \emph{and dually,} $(X\cap Y)\land c\land (X\cap Y)$ \emph{is an} $A$-\emph{coset in} $C$ \emph{for all} $a\in A$, $c\in C$. \emph{Conversely, when} $\mathbf{S}$ \emph{is categorical, given a} $C$-\emph{coset} $U$ \emph{in} $A$, \emph{for all} $b\in B$, $U\land b\land U$ \emph{is an} $AC$-\emph{coset in} $B$; \emph{likewise given any} $A$-\emph{coset} $V$ \emph{in} $C$, $V\lor b\lor V$ \emph{is an} $AC$-\emph{coset in} $B$ \emph{for all} $b\in B$. \emph{In both cases we get the unique} $AC$-\emph{coset in} $B$ \emph{containing} $b$. An extended discussion of these matters occurs in \cite[{\S}2]{JPC}.

We start our characterization of categorical skew lattices in terms of forbidden subalgebras with a relevant lemma.

\begin{lemma}
\lemlabel{3.1}
Let $A > B > C$ be a left-handed skew chain with $a > c\mypara a'>c'$ where $a\neq a'\in A$ and $c\neq c'\in C$. Set $A^* = \{a,a'\}$, $B^* = \{x\in B\mid a > x > c\text{ or }a' > x > c'\}$ and $C^* = \{c,c'\}$. Then $A^* > B^* > C^*$ is a sub-skew chain. In particular,
\begin{enumerate}[label=\emph{\roman*)}]
\item $a'>x>c'$ for $x\in B^*$ implies: $a >$ both $a\land x$ and $x\lor c > c$ with $a\land x-_{A^*} x-_{C^*} x\lor c$.
\item $a>x>c$ for $x\in B^*$ implies: $a'>$ both $a'\land x$ and $x\lor c'>c'$ with $a'\land x-_{A^*} x-_{C^*} x\lor c'$.
\end{enumerate}
All $A^*$-cosets and all $C^*$-cosets in $B^*$ are of order $2$. An $A^* C^*$-component in $B^*$ is either a subset $\{b,b'\}$ that is simultaneously an $A^*$-coset and $C^*$-coset in $B^*$ or else it is a larger subset with all $A^* C^*$-cosets having size $1$ and having the alternating coset form
\[
\cdots -_{A^*} \bullet -_{C^*} \bullet -_{A^*} \bullet -_{C^*} \bullet -_{A^*} \bullet -_{C^*} \cdots
\]
Only the former case can occur if the skew chain is categorical.
\end{lemma}
\begin{proof}
Being left-handed, we need only check the mixed outcomes, say $a\land x$, $x\land a$, $c\lor x$ and $x\lor c$ where $a' > x > c'$ for case (i). Trivially $x\land a = x = c\lor x$. As for $a\land x$, $a\land (a\land x) = a\land x = (a\land x)\land a$, due to left-handedness, so that $a > a\land x$; likewise $c\land (a\land x) = c$, while
\[
(a\land x)\land c = a\land x\land a\land c' = a\land x\land c' = a\land c' = c
\]
by left-handedness and parallelism. Hence $a\land x > c$ also, so that $a\land x$ is in $B^*$. The dual argument gives $a > x\lor c > c$, so that $x\lor c \in B^*$ also. Similarly (ii) holds and we have a sub-skew chain.

Clearly the $A^*$-cosets in $B^*$ either all have order $1$ or all have order $2$. If they have order $1$, then $a,a' >$ all elements in $B^*$, and by transitivity, $a,a' >$ both $c,c'$, so that $a > c$ is not parallel to $a' > c'$. Thus all $A^*$-cosets in $B^*$ have order $2$ and likewise all $C^*$-cosets in $B^*$ have order $2$. In an $A^* C*$-component in $B^*$, if the first case does not occur, a situation $x-_{C^*} y-_{A^*} z$ with $x,y,z$ distinct develops. Since $A^*$-cosets and $C^*$-cosets have size $2$, it extends in an alternating coset pattern in both directions, either doing so indefinitely or eventually connecting to form a cycle of even length.
\end{proof}

A complete set of examples with $B^*$ being a single $A^* C^*$-component is as follows.

\begin{example}
\exmlabel{XY}
Consider the class of skew chains $A > B_n > C$ for $1 \leq n\leq \omega$, where
\begin{align*}
A &= \{a_1,a_2\}, C = \{c_1,c_2\} \text{ and} \\
B_n &= \{b_1,b_2,\ldots,b_{2n}\}\text{ or }\{\ldots,b_{-2},b_{-1},b_0,b_1,b_2,\ldots\}\text{ if }n = \omega\,.
\end{align*}
The partial order is given by parity: $a_1 > b_{\text{odd}} > c_1$ and $a_2 > b_{\text{even}} > c_2$. Both $A$ and $C$ are full $B$-cosets as well as full cosets of each other. $A$-cosets and $C$-cosets in $B$ are given respectively by:
\[
\{b_1,b_2\mid b_3,b_4\mid \cdots\mid b_{2n-1},b_{2n}\}\quad\text{and}\quad
\{b_{2n},b_1\mid b_2,b_3\mid \cdots\mid b_{2n-2},b_{2n-1}\}\quad\text{for }n < \omega\,.
\]
For $n > 1$, $B_n$ has the following alternating coset structure (modulo $n$ when $n$ is finite):
\[
\cdots -_{A} b_{2k-2} -_{C} b_{2k-1} -_{A} b_{2k} -_{C} b_{2k+1} -_{A} b_{2k+2} -_{C} \cdots\,.
\]
Clearly $B_n$ is a single component. We denote the left-handed skew chain thus determined by $\mathbf{X}_n$ and its right-handed dual by $\mathbf{Y}_n$ for $n\leq \omega$. Their Hasse diagrams for $n = 1,2$ are given in Figure 2.
\begin{figure}[htb]
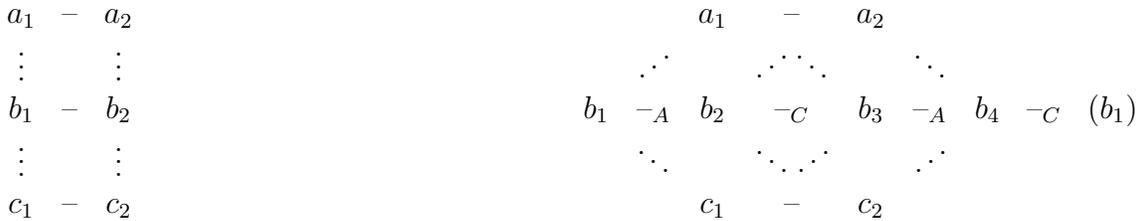

\hfill
\begin{center}
\begin{minipage}{0.4\linewidth}
\begin{tabular}{ccc}
$a_1$    & -- & $a_2$ \\
$\vdots$ &    & $\vdots$ \\
$b_1$    & -- & $b_2$ \\
$\vdots$ &    & $\vdots$ \\
$c_1$    & -- & $c_2$
\end{tabular}
\end{minipage}
\begin{minipage}{0.4\linewidth}
\begin{tabular}{lccccccccr}
      &          & $a_1$    & --     & $a_2$     &           &       &        &         \\
      & $\iddots$ &         & $\iddots \ddots$   &           & $\ddots$ &     &         \\
$b_1$ & --$_A$   & $b_2$    & --$_C$ & $b_3$     & --$_A$    & $b_4$ & --$_C$ & $(b_1)$ \\
      & $\ddots$ &          & $\ddots \iddots$   &            & $\iddots$ &   &         \\
      &          & $c_1$    & --     & $c_2$     &           &       &        &
\end{tabular}
\end{minipage}
\end{center}
\hfill
\caption{Hasse diagrams for $\mathbf{X}_n$/$\mathbf{Y}_n$, $n = 1,2$}
\end{figure}
\end{example}
Applying \eqnref{determine1} and \eqnref{determine2} above, instances of left-handed operations on $\mathbf{X}_2$ are given by
\[
a_1\lor c_2 = a_2 = a_1\lor a_2,\quad a_1\land b_4 = b_3\land b_4 = b_3,\quad\text{and}\quad b_1\lor c_2 = b_1\lor b_4 = b_4\,.
\]
Except for $\mathbf{X}_1$ and $\mathbf{Y}_1$, none of these skew lattices is categorical. In $\mathbf{X}_n$ for $n\geq 2$, $a_1 > b_1 > c_1$, $a_2\land c_1 = c_2$, $a_1\lor c_2 = a_2$, but $a_2\land b_1 = b_2$, while $b_1\lor c_2$ is either $b_{2n}$ or $b_0$. Note that while all $A$-cosets and all $C$-cosets in $B_n$ have order $2$, the $AC$-cosets have order $1$.

\begin{theorem}
\thmlabel{XY}
A left-handed skew lattice is categorical if and only if it contains no copy of $\mathbf{X}_n$ for $2\leq n\leq \omega$. Dually, a right-handed skew lattice is categorical if and only if it contains no copy of $\mathbf{Y}_n$ for $2\leq n\leq \omega$. In general, a skew lattice is categorical if and only if it contains no copy of any of these algebras. Finally, none of these algebras is a subalgebra of another one.
\end{theorem}
\begin{proof}
We begin with a skew chain $A > B > C$ in a left-handed skew lattice $\mathbf{S}$. Given $a > b > c$ in $\mathbf{S}$, where $a\in A$, $b\in B$ and $c\in C$, let $a > c\mypara a' > c'$ with $a\neq a'$. In the skew chain of Lemma \lemref{3.1}, $A^* > B^* > C^*$ where $A^* = \{a,a'\}$ and $C^* = \{c,c'\}$, we obtain the following configuration.
\begin{center}
\begin{tabular}{ccccc}
$a'$ & -- & $a$ & -- & $a'$ \\
$\vdots$ & & $\vdots$ & & $\vdots$ \\
$b\lor c'$ & --$_C$ & $b$ & --$_A$ & $a'\land b$ \\
$\vdots$ & & $\vdots$ & & $\vdots$ \\
$c'$ & -- & $c$ & -- & $c'$
\end{tabular}
\end{center}
When $a'\land c = b\lor c'$, the situation is compatible with $\mathbf{S}$ being categorical. Otherwise, in the $A^* C^*$-component of $b$ in $B^*$, the middle row in the above configuration extends to an alternating coset pattern of the type in Lemma \lemref{3.1}, giving us a copy of $\mathbf{X}_n$ where $2\leq n\leq \omega$. If $\mathbf{S}$ is not categorical, such a situation must occur. Conversely, any left-handed skew lattice containing a copy of $\mathbf{X}_n$ for $n\geq 2$ is not categorical. The first assertion now follows. The nature of the middle row implies that no $\mathbf{X}_m$ can be embedded in any $\mathbf{X}_n$ for $n > m$.

The right-handed case is similar. Clearly, a categorical skew lattice contains no $\mathbf{X}_n$ or $\mathbf{Y}_n$ copy for $n\geq 2$. Conversely, if a skew lattice $\mathbf{S}$ contains copies of none of them, then neither does $\mathbf{S}/\mathcal{R}$ or $\mathbf{S}/\mathcal{L}$ since every skew chain with three $\mathcal{D}$-classes in either $\mathbf{S}/\mathcal{R}$ or $\mathbf{S}/\mathcal{L}$ can be lifted to an isomorphic subalgebra of $\mathbf{S}$. (Indeed, given any skew chain $T$: $A > B > C$, one easily finds $a > b > c$ with $a\in A$, $b\in B$ and $c\in C$. Then, \emph{e.g.}, the sub-skew chain $\mathcal{R}_a > \mathcal{R}_b > \mathcal{R}_c$ of $\mathcal{R}$-classes in $T$ is isomorphic to $T/\mathcal{L}$. See [5].) Thus $\mathbf{S}/\mathcal{R}$ and $\mathbf{S}/\mathcal{L}$ are categorical, and hence so is $\mathbf{S}$.
\end{proof}

A skew chain $A > B > A'$ is \defn{reflective} if (1) $A$ and $A'$ are full cosets of each other in themselves, making $A\equiv A'$ with both being full $B$-cosets in themselves, and (2) $B$ consists of a single $AA'$-component. All $\mathbf{X}_n$ and $\mathbf{Y}_n$ are reflective. If $B$ is both an $A$-coset and an $A'$-coset for every reflective skew chain in a skew lattice $\mathbf{S}$ (making the skew chain a direct product of a chain $a > b > a'$ and a rectangular subalgebra), then $\mathbf{S}$ is categorical. Indeed, copies of $\mathbf{X}_n$ or $\mathbf{Y}_n$ for $n\geq 2$ are eliminated as subalgebras, while $\mathbf{X}_1$ and $\mathbf{Y}_1$ clearly factor as stated.

The converse is also true. Consider a reflective skew chain $A > B > A'$ in a categorical skew lattice. Let $\varphi: A \to B$ be a coset bijection of $A$ onto an $A$-coset in $B$ and let $\psi : B\to A'$ be a coset bijection of $B$ onto $A'$ such that the composition $\psi\circ \varphi$ is the unique coset bijection of $A$ onto $A'$. As partial bijections, the only way for $\psi\circ \varphi$ to be both one-to-one and onto is for $\varphi$ and $\psi$ to be full bijections between $A$ and $B$, and between $B$ and $A'$, respectively, thus making $B$ both a full $A$-coset and a full $A'$-coset within itself. We thus have:

\begin{proposition}
\prplabel{3.4}
A skew lattice $\mathbf{S}$ is categorical if and only if every reflective skew chain $A > B > A'$ in $\mathbf{S}$ factors as a direct product of a chain, $a > b > a'$, and a rectangular skew lattice.
\end{proposition}

\section{Strictly categorical skew lattices}
\seclabel{strict}

Recall that a categorical skew lattice $\mathbf{S}$ is \defn{strictly categorical} if for every skew chain of $\mathcal{D}$-classes $A > B > C$ in $\mathbf{S}$, each $A$-coset in $B$ has nonempty intersection with each $C$-coset in $B$, making both $B$ an entire $AC$-component and empty coset bijections unnecessary. Examples are:
\begin{enumerate}[label=\alph*)]
\item \emph{Normal} skew lattices characterized by the conditions: $x\land y\land z\land w = x\land z\land y\land w$; equivalently, every subset $[e]\downarrow = \{x\in S\mid e\geq x\} = \{e\land x\land e\mid x\in S\}$ is a sublattice;
\item \emph{Conormal} skew lattices satisfying the dual condition $x\lor y\lor z\lor w = x\lor z\lor y\lor w$; equivalently, every subset $[e]\uparrow = \{x\in S\mid e\leq x\} = \{e\lor x\lor e\mid x\in S\}$ is a sublattice;
\item \emph{Primitive} skew lattices consisting of two $\mathcal{D}$-classes: $A > B$ and rectangular skew lattices.
\item \emph{Skew diamonds} in cancellative skew lattices, and in particular, skew diamonds in rings. (A skew diamond is a skew lattice $\{J > A, B > M\}$ consisting of two incomparable $\mathcal{D}$-classes $A$ and $B$ along with their join $\mathcal{D}$-class $J$ and their meet $\mathcal{D}$-class $M$.)  See \cite{KL}.
\end{enumerate}
See \cite{KL} for general results on normal skew lattices. Their importance is due in part to skew Boolean algebras being normal as skew lattices \cite{Bign95,Bign96,Leec96,Leec0x,Spin06}. Some nice counting theorems for categorical and strictly categorical skew lattices are given in \cite{JPC}.

\begin{theorem}
\thmlabel{strict}
Let $A > B > C$ be a strictly categorical skew chain. Then:
\begin{enumerate}[label=\emph{\roman*)}]
\item For any $a\in A$, all images of $a$ in $B$ lie in a unique $C$-coset in $B$;
\item For any $c\in C$, all images of $c$ in $B$ lie in a unique $A$-coset in $B$;
\item Given $a > c$ with $a\in A$ and $c\in C$, a unique $b\in B$ exists such that $a > b > c$. This $b$ lies jointly in the $C$-coset in $B$ containing all images of $a$ in $B$ and in the $A$-coset in $B$ containing all images of $c$ in $B$.
\end{enumerate}
\end{theorem}
\begin{proof}
To verify (i) we assume without loss of generality that $C$ is a full $B$-coset within itself. If $a\land C\land a = \{c\in C\mid a > c\}$ is the image set of $a$ in $C$ parameterizing the $A$-cosets in $C$ and $b\in B$ is such that $a > b$, then $\{c\lor b\lor c\mid c\in a\land C\land a\}$, the set of all images of $a$ in the $C$-coset $C\lor b\lor C$ in $B$, parameterizes the $AC$-cosets in $B$ lying in $C\lor b\lor C$ (since $AC$-cosets in $C\lor b\lor C$ are inverse images of the $A$-cosets in $C$ under the coset bijection of $C\lor b\lor C$ onto $C$). By assumption, all $A$-cosets $X$ in $B$ are in bijective correspondence with all these $AC$-cosets under the map $X\mapsto X\cap C\lor b\lor C$. Thus each element $x$ in $\{c\lor b\lor c\mid c\in a\land C'\land a\}$ is the (necessarily) unique image of $a$ in the $A$-coset in $B$ which $x$ belongs, and as we traverse through these $x$'s, every such $A$-coset occurs as $A\land x\land A$. Thus all images of $a$ in $B$ lie within the $C$-coset $C\lor b\lor C$ in $B$. In similar fashion one verifies (ii). Finally, given $a > c$ with $a\in A$ and $c\in C$, a unique $AC$-coset $U$ exists that is the intersection of the $A$-coset containing all images of $c$ in $B$ and the $C$-coset containing all images of $a$ in $B$. In particular, $U$ contains unique elements $u,v$ such that $a > u$ and $v > c$. Consider $b = a\land v\land a$ in $B$. Clearly $a > b > c$ so that $b$ is a simultaneous image of $a$ and $c$ in $B$ (since $b -_A v$) and thus is in $U$; moreover, by uniqueness of $u$ and $v$ in $U$, we have $u = b = v$.
\end{proof}

This leads to the following multiple characterization of strictly categorical skew lattices.

\begin{theorem}
\thmlabel{4.2}
The following seven conditions on a skew lattice $\mathbf{S}$ are equivalent.
\begin{enumerate}[label=\emph{\roman*)}]
\item $\mathbf{S}$ is strictly categorical;
\item $\mathbf{S}$ satisfies
\[
x > y > z\quad\&\quad x > y' > z\quad\&\quad y\greenD y'
\quad\Rightarrow\quad y = y'\,;
\]
\item $\mathbf{S}$ satisfies
\[
x \geq y \geq z\quad\&\quad x \geq y' \geq z\quad\&\quad y\greenD y'
\quad\Rightarrow\quad y = y'\,;
\]
\item $\mathbf{S}$ has no subalgebra isomorphic to either of the following $4$-element skew chains.
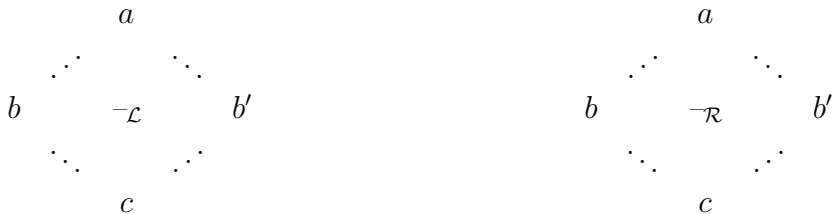
\begin{figure}[htb]
\begin{center}
\begin{minipage}{0.4\linewidth}
\begin{tabular}{rcccl}
    &         & $a$    &         &     \\
    & $\iddots$ &        & $\ddots$  &     \\
$b$ &         & --$_{\mathcal{L}}$ &         & $b'$ \\
    & $\ddots$  &        & $\iddots$ &     \\
    &         & $c$    &         &
\end{tabular}
\end{minipage}
\begin{minipage}{0.4\linewidth}
\begin{tabular}{rcccl}
    &         & $a$    &         &     \\
    & $\iddots$ &        & $\ddots$  &     \\
$b$ &         & --$_{\mathcal{R}}$ &         & $b'$ \\
    & $\ddots$  &        & $\iddots$ &     \\
    &         & $c$    &         &
\end{tabular}
\end{minipage}
\end{center}
\end{figure}
\item If $a > b$ in $\mathbf{S}$, the interval subalgebra $[a,b] = \{x\in S\mid a\geq x\geq b\}$ is a sublattice.
\item Given $a\in S$, $[a]\uparrow = \{x\in S\mid x\geq a\}$ is a normal subalgebra of $\mathbf{S}$ and $[a]\downarrow = \{x\in S\mid a\geq x\}$ is a conormal subalgebra of $S$.
\item $\mathbf{S}$ is categorical and given any skew chain $A > B > C$ of $\mathcal{D}$-classes in $S$, for each coset bijection $\varphi : A\to C$, there exist unique coset bijections $\psi : A\to B$ and $\chi : B\to C$ such that $\varphi = \chi\circ \psi$.
\item Every reflective skew chain $A > B > C$ is an isochain.
\end{enumerate}
\end{theorem}
\begin{proof}
Theorem \thmref{strict}(iii) gives us (i)$\Rightarrow$(ii). Conversely, if $\mathbf{S}$ satisfies (ii) then no subalgebra of $\mathbf{S}$ can be one of the forbidden subalgebras of the last section, making $\mathbf{S}$ categorical. We next show that given $x,y\in B$, there exist $u,v\in B$ such that $x-_A u-_C y$ and $x-_C v -_A y$. This guarantees that in $B$, every $A$-coset meets every $C$-coset. Indeed, pick $a\in A$ and $c\in C$ so that $a > x > c$. Note that $a > a \land (c\lor y\lor c)\land a$, $c\lor (a\land y\land a)\lor c > c$. But by assumption $x$ is the unique element in $B$ between $a$ and $c$ under $>$. Thus $a\land (c\lor y\lor c)\land a = x = c\lor (a\land y\land a)\lor c$ so that both $x-_A c\lor y\lor c -_C y$ and $x-_C a\land y\land a -_A y$ in $B$, which gives (ii)$\Rightarrow$(i).

Next let $\mathbf{S}$ be categorical with $A > B > C$ as stated in (vii). The unique factorization in (vii) occurs precisely when (ii) holds, making (ii) and (vii) equivalent, with (viii) being a variant of (vii). Finally, (iii)-(vi) are easily seen to be equivalent variants of (ii).
\end{proof}

\begin{corollary}
\corlabel{strictvariety}
Strictly categorical skew lattices form a variety of skew lattices.
\end{corollary}
\begin{proof}
We will show that strictly categorical skew lattices are characterized by the following identity (or its dual):
\begin{equation}
\eqnlabel{strictident}
x\lor (y\land z\land u\land y)\lor x = x\lor (y\land u\land z\land y)\lor x\,.
\end{equation}
Let $e$ denote the left side and $f$ denote the right side. Observe that $e\mathcal{D} f$ since $z\land u\greenD u\land z$. Note that $x\lor y\lor x\geq e, f \geq x$ by \eqnref{absorb}. Hence if a skew lattice $\mathbf{S}$ is strictly categorical, then \eqnref{strictident} holds by Theorem \thmref{4.2}(iii). Conversely, let \eqnref{strictident} hold in $\mathbf{S}$ and suppose that $a\geq$ both $b,b'\geq c$ in $S$ with $b\greenD b'$. Assigning $x\mapsto c$, $y\mapsto a$, $z\mapsto b\land b'$ and $u\mapsto b'\land b$ reduced \eqnref{strictident} to $b = b\land b'\land b = b'\land b\land b'$ so that $\mathbf{S}$ is strictly categorical by Theorem \thmref{4.2}(iii).
\end{proof}

While distributive skew lattices are categorical, they need not be strictly categorical, but \emph{a strictly categorical skew lattice} $\mathbf{S}$ \emph{is distributive iff} $\mathbf{S}/\mathcal{D}$ \emph{is distributive.} (See \cite[Theorem 5.4]{KL}.)

It is natural to ask: \emph{What is the variety generated jointly from the varieties of normal and conormal skew lattices?} To refine this question, we first proceed as follows.

A primitive skew lattice $A > B$ is \defn{order-closed} if for $a,a'\in A$ and $b,b'\in B$, both $a,a'> b$ and $a > b,b'$ imply $a' > b'$.
\begin{table}[htb]
\begin{tabular}{lccccc}
$A$ & $a$ & -- & -- & -- & $a'$ \\
    & $\vdots$ & $\ddots$ & & $\iddots$ & \\
    & $\vdots$ & $\iddots$ & & $\ddots$ & \\
$B$ & $b$ & -- & -- & -- & $b'$
\end{tabular}
\end{table}
A primitive skew lattice $A > B$ is \defn{simply order-closed} if $a > b$ for all $a\in A$ and all $b\in B$. In this case the cosets of $A$ and $B$ in each other are singleton subsets. It is easy to verity that \emph{a primitive skew lattice} $\mathbf{S}$ \emph{is order-closed if and only if it factors into a product} $D\times T$ \emph{where} $D$ \emph{is rectangular and} $T$ \emph{is simply order-closed and primitive}.

A skew lattice is \defn{order-closed} if all its primitive subalgebras are thus. Examples include:
\begin{enumerate}[label=\alph*)]
\item Normal skew lattices and conormal skew lattices;
\item The sequences of examples $\mathbf{X}_n$ and $\mathbf{Y}_n$ of section 3.
\end{enumerate}
On the other hand, primitive skew lattices that are not order-closed are easily found. (See \cite[{\S\S}1,2]{Leec93}.)

\begin{theorem}
Order-closed skew lattices form a variety of skew lattices.
\end{theorem}
\begin{proof}
The following generic situation holds between comparable $\mathcal{D}$-classes in a skew lattice:
\begin{table}[htb]
\begin{tabular}{ccccc}
 $x\land y$ & -- & -- & -- & $(x\land y\land u\land v\land x\land y)\lor (y\land x)\lor (x\land y\land u\land v\land x\land y)$ \\
 $\vdots$ & $\ddots$ & & $\iddots$ & \\
 $\vdots$ & $\iddots$ & & $\ddots$ & \\
 $x\land y\land u\land v\land x\land y$ & -- & -- & -- & $x\land y\land v\land u\land x\land y$
\end{tabular}
\end{table}
where as usual, the dotted lines denote $\geq$ relationships. Being order-closed requires both expressions on the right side of the diagram to commute under $\lor$ (or $\land$). Commutativity under $\lor$ together with \eqnref{absorb} gives
\begin{equation}
\eqnlabel{orderident}
(x\land y\land \underbrace{v\land u}\land x\land y)\lor (y\land x)\lor (x\land y\land \underbrace{u\land v}\land x\land y) =
(x\land y\land \underbrace{u\land v}\land x\land y)\lor (y\land x)\lor (x\land y\land \underbrace{v\land u}\land x\land y)
\end{equation}
(or its dual) as a characterizing identity for order-closed skew lattices.
\end{proof}

Refining the above question about the variety generated jointly from the varieties of normal
and conormal skew lattices, we ask:

\begin{problem}
Do order-closed, strictly categorical skew lattices form the join variety of the varieties of normal skew lattices and their conormal duals?
\end{problem}

\end{document}